\documentclass[12pt,a4paper]{article}

\usepackage[left=2.45cm, top=2.45cm,bottom=2.45cm,right=2.45cm]{geometry}

\usepackage{mathtools,amssymb,amsthm,mathrsfs,calc,graphicx,stmaryrd,xcolor,cleveref,dsfont,mathpazo}
\usepackage[british]{babel}
\usepackage{amsfonts}              % for blackboard bold, etc
\newtheoremstyle{definition}
{10pt}% measure of space to leave above the theorem. E.g.: 3pt
{10pt}% measure of space to leave below the theorem. E.g.: 3pt
{}% name of font to use in the body of the theorem
{}% measure of space to indent
{\bfseries}% name of head font
{}% punctuation between head and body
{.5em}% space after theorem head; " " = normal interword space
{}% Manually specify head

\newtheoremstyle{plain}
{10pt}% measure of space to leave above the theorem. E.g.: 3pt
{10pt}% measure of space to leave below the theorem. E.g.: 3pt
{\itshape}% name of font to use in the body of the theorem
{}% measure of space to indent
{\bfseries}% name of head font
{}% punctuation between head and body
{.5em}% space after theorem head; " " = normal interword space
{}% Manually specify head

\theoremstyle{plain}	
\newtheorem{thm}{Theorem}
\newtheorem{lem}[thm]{Lemma}
\newtheorem{prop}[thm]{Proposition}

\theoremstyle{definition}	

\newtheorem{remark}[thm]{Remark}

\DeclarePairedDelimiter\norm{\lVert}{\rVert}

\newcommand{\RR}{\mathbb{R}}      % for Real numbers
\newcommand{\conv}{\operatorname{conv}}
\newcommand{\spa}{\operatorname{span}}
\newcommand{\B}{\operatorname{B}}

\def\BB{\mathbb{B}}

\def\EE{\mathbb{E}}

\def\PP{\mathbb{P}}

\def\RR{\mathbb{R}}
\def\SS{\mathbb{S}}

\def\BBd{\mathbb{B}^d}
\def\SSd{\mathbb{S}^{d-1}}

\def\b{\beta}

\def\k{\kappa}

\def\o{\omega}
\def\G{\Gamma}

\def\bG{\mathbf{G}}

\def\bH{\mathbf{H}}
\def\bB{\mathbf{B}}
\def\bU{\mathbf{U}}

\def\bS{\mathbf{S}}

\def\cH{\mathcal{H}}

\def\F{\textup{F}}

\def\dint{\textup{d}}

\begin{document}

\title{\bfseries Monotonicity of facet numbers \\ of random convex hulls}

\author{Gilles Bonnet\footnotemark[1]\ , Julian Grote\footnotemark[2]\ , Daniel Temesvari\footnotemark[3]\ ,\\ Christoph Th\"ale\footnotemark[4]\ , Nicola Turchi\footnotemark[5]\ \ and Florian Wespi\footnotemark[6]}

\date{}
\renewcommand{\thefootnote}{\fnsymbol{footnote}}
\footnotetext[1]{Faculty of Mathematics, Ruhr University Bochum, Germany. E-mail: gilles.bonnet@rub.de}

\footnotetext[2]{Faculty of Mathematics, Ruhr University Bochum, Germany. E-mail: julian.grote@rub.de}

\footnotetext[3]{Faculty of Mathematics, Ruhr University Bochum, Germany. E-mail: daniel.temesvari@rub.de}

\footnotetext[4]{Faculty of Mathematics, Ruhr University Bochum, Germany. E-mail: christoph.thaele@rub.de}

\footnotetext[5]{Faculty of Mathematics, Ruhr University Bochum, Germany. E-mail: nicola.turchi@rub.de}

\footnotetext[6]{Institute of Mathematical Statistics and Actuarial Science, University of Bern, Switzerland. E-mail: florian.wespi@stat.unibe.ch}

%\footnotetext[7]{This work has been supported by the German Research Foundation DFG via GRK2131.}

\maketitle

\begin{abstract}
Let $X_1,\ldots,X_n$ be independent random points that are distributed according to a probability measure on $\mathbb{R}^d$ and let $P_n$ be the random convex hull generated by $X_1,\ldots,X_n$ ($n\geq d+1$). Natural classes of probability distributions are characterized for which, by means of  Blaschke-Petkantschin formulae from integral geometry, one can show that the mean facet number of $P_n$ is strictly monotonically increasing in $n$.
\end{abstract}

\section{Introduction and main result}

Fix a space dimension $d\geq2$. For an integer $n\geq d+1$, let $X_1,\dots,X_n$ be independent random points that are chosen according to an absolutely continuous probability distribution on $\RR^d$. By $P_{n-1}$ and $P_n$ we denote the random convex hulls generated by $X_1,\ldots,X_{n-1}$ and $X_1,\ldots,X_n$, respectively. In our present text we are interested in the mean number of facets $\EE f_{d-1}(P_{n-1})$ and $\EE f_{d-1}(P_n)$ of $P_{n-1}$ and $P_n$. More specifically, we ask the following monotonicity question:
\begin{center}
\textit{Is it true that $\EE f_{d-1}(P_{n-1})\leq\EE f_{d-1}(P_n)$?}
\end{center}
This question has been put forward and answered positively by Devillers, Glisse, Goa\-oc, Moroz and Reitzner \cite{DevillersEtAl} for random points that are uniformly distributed in a convex body $K\subset\RR^d$ if $d=2$ and, if $d\geq 3$, under the additional assumptions that the boundary of $K$ is twice differentiable with strictly positive Gaussian curvature and that $n$ is sufficiently large, that is, $n\geq n(K)$, where $n(K)$ is a constant depending on $K$. Moreover, an affirmative answer was obtained by Beermann \cite{Beer} if the random points are chosen with respect to the standard Gaussian distribution on $\RR^d$ or according to the uniform distribution in the $d$-dimensional unit ball for all $d\geq 2$. Beermann's proof essentially relies on a Blaschke-Petkantschin formula, a well known change-of-variables formula in integral geometry. Our aim in this text is to generalize her approach to other and more general probability distributions on $\RR^d$. In fact, we will be able to characterize all absolutely continuous rotationally symmetric distributions on $\RR^d$ whose densities satisfy a natural scaling property (see \eqref{eq:Scaling} below), to which the methodology based on the Blaschke-Petkantschin formula  can be applied and for which we can answer positively the monotonicity question posed above for any of these distributions. Moreover, we will apply our results to study similar monotonicity questions for a class of spherical convex hulls generated by random points on a half-sphere, which comprises as a special case the model recently studied by B\'ar\'any, Hug, Reitzner and Schneider \cite{BaranyHugReitznerSchneider}.

\medskip

To present our main result formally, we introduce four classes of probability measures:
\begin{itemize}
\item[-] $\bG$ is the class of centred Gaussian distributions on $\RR^d$ with density proportional to $$x\mapsto \exp\Big(-{\|x\|^2\over 2\sigma^2}\Big),\qquad \sigma>0,$$
\item[-] $\bH$ is the class of heavy-tailed distributions on $\RR^d$ with density proportional to $$x\mapsto \Big(1+{\|x\|^2\over\sigma^2}\Big)^{-\beta},\qquad \beta>d/2,\sigma>0,$$
\item[-] $\bB$ is the class of beta-type distributions on the $d$-dimensional centred ball $\BB^d_\sigma$ of radius $\sigma$ with density proportional to $$x\mapsto \Big(1-{\|x\|^2\over\sigma^2}\Big)^\beta,\qquad\beta>-1,\sigma>0,$$
\item[-] $\bU$ comprises the uniform distributions on the $(d-1)$-dimensional centred spheres $\SSd_\sigma$ with radius $\sigma>0$.
\end{itemize}
It will turn out that the classes $\bG$, $\bH$, $\bB$ and $\bU$ contain precisely the absolutely continuous rotationally symmetric probability distributions on $\RR^d$, whose densities satisfy the natural scaling property \eqref{eq:Scaling} below, for which monotonicity of the mean facet number of the associated random convex hulls can be shown by means of arguments based on a Blaschke-Petkantschin formula, see the discussion at the end of Section \ref{sec:Proofmonotonicity} for further details. In fact, our result shows that even the stronger strict monotonicity holds.

\begin{thm}\label{thm:monotonicity}
Let $X_1,\dots,X_n \in \RR^d$, $n\geq d+1$, be independent and identically distributed according to a probability measure belonging to one of the classes $\bG$, $\bH$, $\bB$ or $\bU$. Then,
\[ \EE f_{d-1} (P_n) > \EE f_{d-1} (P_{n-1}).\]
\end{thm}

It should be emphasized that strict monotonicity of $n\mapsto f_{d-1}(P_n)$ cannot hold pathwise (except for the trivial case $n=d+1$), since the addition of a further random point can reduce the facet number arbitrarily as the additional point might `see' much more than $d$ vertices of the already constructed random convex hull. For this reason, the expectation in Theorem \ref{thm:monotonicity} is essential to deduce a non-trivial result.

We would also like to remark that monotonicity questions related to the volume of random convex hulls have recently attracted some interest in convex geometry because of their connection to the famous slicing problem. Namely, if $K$ and $L$ are two compact convex sets in $\RR^d$ with interior points, let $V_K$ and $V_L$ be the volume of the convex hull of $d+1$ independent random points uniformly distributed in $K$ and $L$, respectively. One is interested in the question whether the set inclusion $K\subseteq L$ implies the inequality $\EE V_K\leq \EE V_L$. In particular, the work of Rademacher \cite{Rademacher} shows that this is false in general whenever $d\geq 4$. Higher moments were treated by Reichenwallner and Reitzner \cite{ReichenwallnerReitzner}, and we refer to the discussion therein for further details and background material.

\medskip

Our text is structured as follows. In Section \ref{sec:Background} we recall the necessary background material from integral geometry and introduce some more notation. Moreover, in Section \ref{sec:aux-results} we develop several auxiliary results that prepare for the proof of Theorem \ref{thm:monotonicity}, which is presented in Section \ref{sec:Proofmonotonicity}. We also discuss in Section \ref{sec:Proofmonotonicity} the limitations and possible extensions of the method we use. The final Section \ref{sec:halfsphere} contains the application of Theorem \ref{thm:monotonicity} to random convex hulls on a half-sphere.

\section{Background material from integral geometry}\label{sec:Background}

\subsection{General notation}
We denote by $A(d,q)$ the Grassmannian of all $q$-dimensional affine subspaces of $\RR^d$, where $q\in\{0,1,\ldots,d\}$. It is a locally compact, homogeneous space with respect to the group of Euclidean motions in $\RR^d$. The corresponding locally finite, motion invariant measure is denoted by $\mu_q$, which is normalized in such a way that
\begin{displaymath}
  \mu_q(\{E\in A(d,q):E \cap \BBd \neq \emptyset\})=\kappa_{d-q},
\end{displaymath}
see \cite{SW}. Here, $\BBd$ is the centred $d$-dimensional unit ball, $\kappa_{d-q}={\pi^{(d-q)/2}\over\Gamma(1+{d-q\over 2})}$ is the $(d-q)$-dimensional volume of $\BB^{d-q}$ and $\Gamma(\,\cdot\,)$ indicates the Gamma function. Moreover, the Beta function is given by 
$$
B(a,b):=\int_0^1s^{a-1}(1-s)^{b-1}\,\dint s={\Gamma(a)\Gamma(b)\over\Gamma(a+b)},\qquad a,b>0.
$$
We shall denote by $\SSd$ the $(d-1)$-dimensional unit sphere and abbreviate by $\o_d=d\k_d$ its total spherical Lebesgue measure. For a subspace $E\in A(d,q)$, we let $\lambda_E$ be the Lebesgue measure on $E$.
%Note that each $E\in A(d,d-1)$ is an affine hyperplane in $\RR^d$.

For a set $K\subset\RR^d$, we shall write $\mathcal{H}_K^{q}$ for the $q$-dimensional Hausdorff measure on $K$. The operators $\conv(\,\cdot\,)$ and $\spa(\,\cdot\,)$ denote the convex and linear hull of the argument. We will use the notation $\Delta_{d-1}(x_1,\dots,x_d)$ to indicate the $(d-1)$-dimensional volume of the convex hull of $d$ points $x_1,\dots,x_d$.

\subsection{Blaschke-Petkantschin formulae} 
Our proof of Theorem \ref{thm:monotonicity} heavily relies on Blaschke-Petkant\-schin formulae from integral geometry. First, we rephrase a special case of the affine Blaschke-Petkantschin formula in $\RR^d$, which appears as Theorem 7.2.7 in \cite{SW}.

\begin{prop}\label{thm:affine-BP}\
	Let $f: (\RR^d)^d \rightarrow \RR$ be a non-negative measurable function. Then,
	\begin{align*}
		&\int \limits_{(\RR^d)^d} f(x_1,\dots,x_d) \,\dint(x_1,\dots,x_d)\\ 
		&= \frac{\omega_d}{2} (d-1)! \int \limits_{A(d,d-1)} \int \limits_{H^d} f(x_1,\dots,x_d) \Delta_{d-1}(x_1,\dots,x_d)\, \lambda_H^d(\dint(x_1,\dots,x_d)) \,\mu_{d-1}(\dint H).
	\end{align*}
\end{prop}

Besides the affine Blaschke-Petkantschin formula in $\RR^d$ we need its spherical counterpart, which is a special case of Theorem 1 in \cite{Zaehle} and can also be found in \cite[Theorem 4]{Miles}.

\begin{prop}\label{thm:spherical-affine-BP}
	Let $f:(\SSd)^d \rightarrow \RR$ be a non-negative measurable function. Then,
	\begin{align*}
		&\int \limits_{(\SSd)^d} f(x_1,\dots,x_d) \, \mathcal{H}^{d(d-1)}_{(\SSd)^d} (\dint(x_1,\ldots,x_d))= (d-1)! \int \limits_{A(d,d-1)} \int \limits_{(H\cap\SSd)^d} f(x_1,\dots,x_d) \\
		&\qquad\qquad\qquad\times\Delta_{d-1}(x_1,\dots,x_d) (1-h^2)^{-\frac{d}{2}} \, \mathcal{H}^{d(d-2)}_{(H \cap\SSd)^d} (\dint(x_1,\ldots,x_d)) \,\mu_{d-1} (\dint H) \mbox{,}
	\end{align*}
	where $h$ denotes the distance from $H$ to the origin.
\end{prop}

\subsection{A slice integration formula}

Finally, we will make use of the following special case of the spherical slice integration formula taken from Theorem A.4 in \cite{ABR}.

\begin{prop}\label{thm:slice-integration}
	Let $f:\SSd \rightarrow \RR$ be a non-negative measurable function. Then,
	\[ 
		\int \limits_{\SSd} f(x)\,\mathcal{H}^{d-1}_{\SSd}(\dint x) = \int \limits_{-1}^1 (1-t^2)^{\frac{d-3}{2}} \int \limits_{\SS^{d-2}} f(t,\sqrt{1-t^2}\,y)\, \mathcal{H}^{d-2}_{\SS^{d-2}}(\dint y) \,\dint t \mbox{.}
	 \]
\end{prop}

\section{Auxiliary results}\label{sec:aux-results}

\subsection{An estimate for integrals of concave functions}
A version of the next lemma was already stated in \cite{Beer}, but without proof. For the sake of completeness we include here the short argument.

\begin{lem}\label{l1}
	Let $h:\RR \rightarrow \RR$ be a non-negative measurable function which is strictly positive on a set of positive Lebesgue measure. Further, let $g:\RR \rightarrow \RR$ be a linear function with negative slope and root $s^\ast \in [0,1]$. Moreover, let $L:\RR \rightarrow \RR$ be positive and strictly concave on $[0,1]$. Then,
	\begin{equation}\label{e1}
		\int \limits_{0}^{1} h(s)g(s)L(s)^{d-1} \,\dint s > \int \limits_0^1 h(s)g(s)\ell(s)^{d-1} \,\dint s \mbox{,}
	\end{equation}
	where $\ell(s)=\frac{L(s^\ast)}{s^\ast}s$.
\end{lem}

\begin{proof}
	We start by exploiting the positivity and strict concavity of $L$. For $s \in [0,s^\ast)$, it implies that
	\begin{equation}\label{eq:conc1}
L(s) =  L\left( \frac{s}{s^\ast} s^\ast \right) >  \frac{s}{s^\ast} L(s^\ast),
	\end{equation}
	while for $s \in (s^\ast,1]$, it gives
	\begin{equation}\label{eq:conc2}
L(s) <  \frac{s}{s^\ast} L(s^\ast).
	\end{equation}
	Since $g$ has a negative slope, it is positive on $[0,s^\ast)$ and negative on $(s^\ast,1]$. Splitting the integral on the left hand side of \eqref{e1} at the point $s^\ast$ and using \eqref{eq:conc1} and \eqref{eq:conc2}, respectively, yields
	\begin{align*}
		&\int \limits_{0}^{1} h(s)g(s)L(s)^{d-1} \, \dint s\\
		 &= \int \limits_{0}^{s^\ast} h(s)g(s)L(s)^{d-1} \, \dint s + \int \limits_{s^\ast}^{1} h(s)g(s)L(s)^{d-1} \, \dint s \\
		& > \int \limits_{0}^{s^\ast} h(s)g(s)\left(\frac{s}{s^\ast} L(s^\ast)\right)^{d-1} \, \dint s + \int \limits_{s^\ast}^{1} h(s)g(s)\left(\frac{s}{s^\ast} L(s^\ast)\right)^{d-1} \, \dint s \\
		& = \int \limits_{0}^{1} h(s)g(s) \ell(s)^{d-1}  \,\dint s.
	\end{align*}
	This completes the argument.
\end{proof}

\subsection{Computation of marginal densities}

Recall the definitions of the distribution classes $\bH$, $\bB$ and $\bU$. It will turn out that it suffices to consider the cases where the scale parameters $\sigma$ there are equal to $1$. From now on we restrict to these cases and denote the density of a distribution in $\bH$ by
\begin{equation}\label{eq:DensityH}
p_{\bH,\beta}(x)=\pi^{-d/2}\frac{\G(\b)}{\G(\b-{d\over 2})} (1+\norm{x}^2)^{-\b}\,,\qquad x\in\RR^d\,,\beta>d/2\,,
\end{equation}
that of a distribution in $\bB$ by
$$
p_{\bB,\beta}(x)=\pi^{-d/2}\frac{\G({d\over 2}+\b+1)}{\G(\b+1)}(1-\norm{x}^2)^\b\,,\qquad x\in\BBd\,,\beta>-1\,,
$$
and note that the uniform distribution on $\SSd$ has density
$$
p_{\bU}(x)=\frac{1}{\omega_d}\,,\qquad x\in\SSd\,,
$$
with respect to the spherical Lebesgue measure. The next lemma provides formulas for the densities of the one-dimensional marginals of these distributions and shows, that the classes $\bB$ and $\bH$ are in some sense closed under one-dimensional projections. Since all distributions we consider are rotationally symmetric, it is sufficient to consider projections onto the first coordinate. We would like to emphasize that the proof of Lemma \ref{l2} uses in an essential way the scaling property \eqref{eq:Scaling} below of the involved densities.

\begin{lem}\label{l2}
	Let $\Pi:\RR^d \rightarrow\RR,(x_1,\ldots,x_d)\mapsto x_1$ be the  projection onto the first coordinate.
	\begin{itemize}
	\item[(i)] Let $\PP\in\bH$ be a distribution with density $p_{\bH,\beta}$ for some $\beta>d/2$. Then, the image measure of $\PP$ under $\Pi$ has density
	$$
	f_{\bH,\beta}(x)=\pi^{-1/2}\frac{\G\left(\b-\frac{d-1}{2}\right)}{\G\left(\b-\frac{d}{2}\right)}(1+x^2)^{\frac{d-1}{2}-\b},\qquad x\in\RR.
	$$
	\item[(ii)] Let $\PP\in\bB$ be a distribution with density $p_{\bB,\beta}$ for some $\beta>-1$. Then, the image measure of $\PP$ under $\Pi$ has density
	$$
	f_{\bB,\beta}(x)=\pi^{-1/2}\frac{\G\left(\b+1+\frac{d}{2}\right)}{\G\left(\b+\frac{d+1}{2}\right)}(1-x^2)^{\frac{d-1}{2}+\b},\qquad x\in[-1,1].
	$$	
	\item[(iii)] Let $\PP\in\bU$ be the uniform distribution on $\SSd$. Then, the image measure of $\PP$ under $\Pi$ has density
	$$
	f_{\bU}(x)=\pi^{-1/2}\frac{\G\left( \frac{d}{2} \right)}{\G\left(\frac{d-1}{2}\right)}(1-x^2)^{\frac{d-3}{2}},\qquad x\in[-1,1].
	$$
	\end{itemize}
\end{lem}

\begin{proof}
To prove (i) we put $x=(x_1,\ldots,x_d)\in\RR^d$, $y:=(x_2,\dots,x_d)$ and also define $c_{\bH,d,\b}:=\pi^{-d/2}\frac{\G(\b)}{\G(\b-d/2)}$. Then, 
	\begin{align*}
		&\int \limits_{\RR^{d-1}} c_{\bH,d,\b}  \left(1+\norm{x}^2\right)^{-\b} \, \dint (x_2, \dots, x_d) \\
		&= \int \limits_{\RR^{d-1}} c_{\bH,d,\b} (1+x_1^2)^{-\b} \left(1+\frac{\norm{y}^2}{1+x_1^2}\right)^{-\b}  \,\dint y \\
		&= (1+x_1^2)^{-\b} \int \limits_{\RR^{d-1}} c_{\bH,d,\b}  \left(1+\norm{z}^2\right)^{-\b} (1+x_1^2)^{\frac{d-1}{2}} \, \dint z \\
		&=(1+x_1^2)^{\frac{d-1}{2}-\b} \frac{c_{\bH,d,\b}}{c_{\bH,d-1,\b}} \int \limits_{\RR^{d-1}} c_{\bH,d-1,\b}  \left(1+\norm{z}^2\right)^{-\b} \, \dint z \\
		&= \frac{c_{\bH,d,\b}}{c_{\bH,d-1,\b}} (1+x_1^2)^{\frac{d-1}{2}-\b} \mbox{,}
	\end{align*}
	where we used the substitution $z=y/\sqrt{1+x_1^2}$. Plugging in the constants yields the desired result.

Next, we consider the distribution with density $p_{\bB,\beta}$. For $x=(x_1,\ldots,x_d)\in\BB^d$, we put again $y:=(x_2,\dots,x_d)$ and abbreviate $c_{\bB,d,\b}:=\pi^{-d/2}\frac{\G({d\over 2}+\b+1)}{\G(\b+1)}$. Then, similarly as above, we compute
	\begin{align*}
		&\int \limits_{\BB^{d-1}} c_{\bB,d,\b}  \left(1-\norm{x}^2\right)^{\b} \, \dint (x_2, \dots, x_d) \\
		&= \int \limits_{\BB^{d-1}} c_{\bB,d,\b} (1-x_1^2)^{\b} \left(1-\frac{\norm{y}^2}{1-x_1^2}\right)^{\b} \, \dint y \\
		&= (1-x_1^2)^{\b} \int \limits_{\BB^{d-1}} c_{\bB,d,\b}  \left(1-\norm{z}^2\right)^{\b} (1-x_1^2)^{\frac{d-1}{2}} \, \dint z \\
		&=(1-x_1^2)^{\frac{d-1}{2}+\b} \frac{c_{\bB,d,\b}}{c_{\bB,d-1,\b}} \int \limits_{\BB^{d-1}} c_{\bB,d-1,\b}  \left(1-\norm{z}^2\right)^{\b} \, \dint z \\
		&= \frac{c_{\bB,d,\b}}{c_{\bB,d-1,\b}} (1-x_1^2)^{\frac{d-1}{2}+\b} \mbox{,}
	\end{align*}
	where we used the substitution $z=y/\sqrt{1-x_1^2}$. Again, simplification of the constants yields the desired result.

Finally, we consider the case of the uniform distribution on $\SSd$. We denote by $F$ the distribution function of the image measure of $\PP=\o_d^{-1}\cH^{d-1}_{\SSd}$ under the orthogonal projection map $\Pi$ and let $x_1 \in [-1,1]$. Using the slice integration formula from \Cref{thm:slice-integration}, we obtain
	\begin{align*}
		\F(x_1) &= \frac{1}{\omega_d}\mathcal{H}_{\SSd}^{d-1} \left( \left\{u \in\SSd : \Pi(u) \in [-1,x_1]\right\} \right) \\
		&= \frac{1}{\omega_d}\int \limits_{\SSd} \mathds{1} \{ \Pi(u) \in [-1,x_1] \} \, \mathcal{H}_{\SSd}^{d-1} (\dint u) \\
		&=\frac{1}{\omega_d} \int \limits_{-1}^{x_1} (1-t^2)^{\frac{d-3}{2}} \int \limits_{\SS^{d-2}}  \, \mathcal{H}^{d-2}_{\SS^{d-2}}(\dint y) \,\dint t \\
		&= \frac{\o_{d-1}}{\o_d} \int \limits_{-1}^{x_1} (1-t^2)^{\frac{d-3}{2}}  \,\dint t \mbox{.}
	\end{align*} 	
Differentiation with respect to $x_1$, together with the definitions of $\o_d$ and $\o_{d-1}$, complete the proof.
\end{proof}

In what follows, we shall denote by $F_{\bH,\beta}$, $F_{\bB,\beta}$ and $F_\bU$ the distribution functions corresponding to the densities $f_{\bH,\beta}$, $f_{\bB,\beta}$ and $f_\bU$ computed in Lemma \ref{l2}, respectively.

\begin{remark}
The marginal densities of the Gaussian distributions belonging to the class $\bG$ can also be computed along the lines of the proof of Lemma \ref{l2}. This yields one-dimensional Gaussian marginals. Since random convex hulls of Gaussian points have already been treated in \cite{Beer}, we decided to concentrate on the classes $\bH$, $\bB$ and $\bU$.
\end{remark}

\section{Proof of Theorem \ref{thm:monotonicity} and discussion}\label{sec:Proofmonotonicity}

\subsection{Proof of Theorem \ref{thm:monotonicity}}

Based on the results from the two previous sections we are now able to present the proof of our main result. In what follows, we denote all constants by $c$. Unless otherwise stated, they only depend on the space dimension $d$ and the parameter $\beta$ of the underlying probability distribution. Their value might change from instance to instance.

\begin{proof}[Proof of \Cref{thm:monotonicity}.]
For the classes $\bG$, $\bH$, $\bB$ and $\bU$ it is sufficient to consider the case that the scale parameter $\sigma$ is equal to $1$, since the mean facet number is invariant under rescalings.

\medskip

We start with the distributions from class $\bG$. However, since this case has already been treated in \cite{Beer}, we refer to Theorem 5.3.1 there.

\medskip

Next, we consider the heavy-tailed distribution on $\RR^d$ with density $p_{\bH,\beta}(x)=c_{\bH,d,\b} (1+\norm{x}^2)^{-\b}$, where $\b>d/2$ and $c_{\bH,d,\b}=\pi^{-d/2}\frac{\G(\b)}{\G(\b-d/2)}$. We follow the ideas from \cite{Beer} and start with the equality
	\begin{equation}\label{e2}
	\begin{split}
		\EE f_{d-1}(P_n) &= \EE\sum_{1\leq i_1<\ldots<i_d\leq n}\mathds{1}\{\conv(X_{i_1},\ldots,X_{i_d})\text{ is a facet of }P_n\}\\
		&=\binom{n}{d} \PP(\conv(X_1,\dots,X_d) \text{ is a facet of }P_n) ,
	\end{split}		
	\end{equation}
	which holds due to the fact that the random points $X_1,\ldots,X_n$ are independent and identically distributed. Let us denote by $H\in A(d,d-1)$ the affine hull of the $(d-1)$-dimensional simplex $P_d$ spanned by $X_1,\ldots,X_d$.
	In the case that $P_d$ is a facet of $P_n$, all the remaining points $X_{d+1},\dots,X_n$ have to lie in one of the (open) half-spaces determined by $H$. If we denote by $\Pi_H$ the orthogonal projection onto $H^\perp$, the orthogonal complement of $H$, we observe that $P_d$ is a facet of $P_n$ if and only if the point $\Pi_H(P_d)$ is not contained in the interior of the interval $\Pi_H(P_n)$ on $H^\perp$. Therefore, using \Cref{l2}, the affine Blaschke-Petkantschin formula from \Cref{thm:affine-BP} and the abbreviation $F^\ast=F_{\bH,\beta}(\Pi_H(P_d))$, we get for the probability that $P_d$ is a facet of $P_n$,
	\begin{align*}
		&\PP(\conv(X_1,\dots,X_d) \text{ is a facet of }P_n ) \\ &= \int \limits_{(\RR^d)^d} \left( (1-F^\ast)^{n-d} + (F^\ast)^{n-d} \right) \prod_{i=1}^{d} c_{\bH,d,\b}  (1+\norm{x_i}^2)^{-\b} \, \dint(x_1,\dots,x_d) \\
		&=c \int \limits_{A(d,d-1)} \int \limits_{H^d} \left( (1-F^\ast)^{n-d} + (F^\ast)^{n-d} \right) \Delta_{d-1}(x_1,\dots,x_d)  \prod_{i=1}^{d} c_{\bH,d,\b}  (1+\norm{x_i}^2)^{-\b} \\
		&\qquad\qquad\times	 \lambda_H^d(\dint(x_1,\ldots,x_d)) \,\mu_{d-1}(\dint H) \mbox{.}
	\end{align*}
	Next, we use the theorem of Pythagoras to decompose, for each $i\in\{1,\ldots,d\}$, the norm $\|x_i\|$. Namely, writing $\|\cdot\|_H$ for the Euclidean norm in $H\in A(d,d-1)$ and $h$ for the distance from $H$ to the origin in $\RR^d$, we have that
	\begin{equation*}
	\norm{x_i}^2= \norm{x_i}^2_H + h^2 \mbox{.}
	\end{equation*}
	Therefore and as already used in the proof of Lemma \ref{l2}, the last term of the integrand can be rewritten as
	\begin{equation}\label{eq:Pythagoras}
(1+\norm{x_i}^2)^{-\b} = (1+h^2+ \norm{x_i}^2_H)^{-\b} = (1+h^2)^{-\b}\left( 1+ \frac{\norm{x_i}^2_H}{1+h^2} \right)^{-\b}.
	\end{equation}
	Moreover, since each hyperplane $H=H(u,h)$ is uniquely determined by its unit normal vector $u\in\SSd$ and its distance $h\in[0,\infty)$ to the origin, the integration over $A(d,d-1)$ can be replaced by a twofold integral over $\SSd$ and $[0,\infty)$. 	Using the substitutions $y_i=x_i/\sqrt{1+h^2}$ with $\lambda_H(\dint x_i) = (1+h^2)^{(d-1)/2} \lambda_H(\dint y_i)$, the rotational invariance of the underlying probability measure, and writing $F(h)$ for $F_{\bH,\beta}(h)$ as well as $f(h)$ for $f_{\bH,\beta}(h)$, gives in view of Lemma \ref{l2} that 
	\begin{align*}
		&\PP(\conv(X_1,\dots,X_d)  \text{ is a facet of }P_n ) \\ 
		&=c \int \limits_{\SSd} \int \limits_0^\infty  \int \limits_{H^d} \left( (1-F(h))^{n-d} + F(h)^{n-d} \right) \Delta_{d-1}(x_1,\dots,x_d)  \\ 
		&\qquad\qquad\times	(1+h^2)^{-d\b} \prod_{i=1}^{d} c_{\bH,d,\b}  \left( 1+ \frac{\norm{x_i}^2_H}{1+h^2} \right)^{-\b}  \,\lambda_H^d(\dint(x_1,\ldots,x_d))\,\dint h \, \mathcal{H}^{d-1}_{\SSd}(\dint u) \\
		& =c \int \limits_{\SSd} \int \limits_0^\infty \left( (1-F(h))^{n-d} + F(h)^{n-d} \right) (1+h^2)^{-d\left(\beta-\frac{d-1}{2}\right)+\frac{d-1}{2}} \,\dint h \,\mathcal{H}^{d-1}_{\SSd}(\dint u) \\
		&\qquad\qquad\times \int \limits_{H^d}  \Delta_{d-1}(y_1,\dots,y_d) \prod_{i=1}^{d} c_{\bH,d-1,\b}  \left( 1+ \|y_i\|^2_H \right)^{-\b} \,\lambda_H^d(\dint(y_1,\ldots,y_d)) \\
		& =c \int \limits_0^\infty \left( (1-F(h))^{n-d} + F(h)^{n-d} \right) (1+h^2)^{-d\left(\beta-\frac{d-1}{2}\right)+\frac{d-1}{2}}  \,\dint h \\
		& = c \int \limits_{-\infty}^\infty (1-F(h))^{n-d} f(h)^d (1+h^2)^{\frac{d-1}{2}} \,\dint h \mbox{,}
	\end{align*}
	where we also used the fact that the integral over $H^d$ is a finite constant given by Equation (72) in \cite{Miles} and which only depends on the space dimension $d$ and on $\beta$.
	
	Write now $s=F(h)$ and $L(s)=f\left(F^{-1}(s)\right)  \sqrt{ 1+(F^{-1}(s))^2 } $ to obtain
	\[ \PP(\conv(X_1,\dots,X_d)  \text{ is a facet of }P_n  ) = c \int \limits_0^1 (1-s)^{n-d} L(s)^{d-1} \,\dint s \mbox{.} \]
	Thus, combination of the above computation with \eqref{e2} yields the representation
	\begin{equation}\label{e3}
	\begin{split}
		&\EE f_{d-1}(P_n) - \EE f_{d-1}(P_{n-1}) \\
		&\qquad\qquad= c \int \limits_0^1 \left[ \binom{n}{d}(1-s) - \binom{n-1}{d} \right] (1-s)^{n-d-1}L(s)^{d-1} \,\dint s \mbox{.}
	\end{split}
	\end{equation}

	In order to apply \Cref{l1}, we have to verify that $L(s)$ is strictly concave on $(0,1)$. We prove this by showing that the second derivative of $L(s)$ is negative. So, let $c_{\bH,1,\b}:=\pi^{-1/2} \frac{\G(\b)}{\G(\b-1/2)}$ and recall that $f(x)=c_{\bH,1,\b} (1+x^2)^{\frac{d-1}{2}-\b}$ from Lemma \ref{l2}. Furthermore, from the definition of $F$ it follows that
	\begin{equation} \label{e4}
	\left(F^{-1}(s)\right)^\prime=\frac{1}{f\left( F^{-1}(s) \right)}=\frac{1}{c_{\bH,1,\b} \left( 1 + \left( F^{-1}(s) \right)^2 \right)^{\frac{d-1}{2}-\b}}.
	\end{equation}
	We recall that
	\begin{align*}
		L(s)&=f\left(F^{-1}(s)\right)  \sqrt{ 1+(F^{-1}(s))^2 } =c_{\bH,1,\b}\left( 1+(F^{-1}(s))^2 \right)^{\frac{d}{2}-\b} \mbox{.}
	\end{align*}
	Hence, using \eqref{e4}, the first derivative of $L(s)$ is
	\begin{align*}
		L^\prime(s)&=c_{\bH,1,\b}\left( \frac{d}{2} - \b \right) \left( 1+(F^{-1}(s))^2 \right)^{\frac{d-2}{2}-\b} 2 F^{-1}(s) \left(F^{-1}(s)\right)^\prime \\
		&= 2 \left( \frac{d}{2} - \b \right) \left( 1+(F^{-1}(s))^2 \right)^{-\frac{1}{2}} F^{-1}(s)
	\end{align*}
	and, thus, for the second derivative we find that
	\begin{align*}
		L^{\prime\prime}(s)&=2 \left( \frac{d}{2} - \b \right)\left[  \left( 1+(F^{-1}(s))^2 \right)^{-\frac{1}{2}} \left(F^{-1}(s)\right)^\prime\right.\\
		&\qquad\qquad\left.- \frac{1}{2} \left( 1+(F^{-1}(s))^2 \right)^{-\frac{3}{2}} 2 \left(F^{-1}(s)\right)^2 \left(F^{-1}(s)\right)^\prime   \right] \\
		&= 2 \left( \frac{d}{2} - \b \right) \left(F^{-1}(s)\right)^\prime \left[  \left( 1+(F^{-1}(s))^2 \right)^{-\frac{1}{2}} - \left( 1+(F^{-1}(s))^2 \right)^{-\frac{3}{2}} \left(F^{-1}(s)\right)^2  \right] \\
		&= \frac{2}{c_{\bH,1,\b}} \left( \frac{d}{2} - \b \right) \left( 1+(F^{-1}(s))^2 \right)^{\b-1-\frac{d}{2}} \left[ 1 + (F^{-1}(s))^2 -(F^{-1}(s))^2   \right] \\
		&= -\frac{2}{c_{\bH,1,\b}}  \left( \b - \frac{d}{2} \right) \left( 1+(F^{-1}(s))^2 \right)^{\b-1-\frac{d}{2}} \\
		&<0,
	\end{align*}
	where the last inequality follows from the fact that $\b>d/2$. As a consequence, we can apply \Cref{l1} to deduce that
	\begin{align*}
		&\EE f_{d-1}(P_n) - \EE f_{d-1}(P_{n-1})\\
		 &= c \int \limits_0^1 \left[ \binom{n}{d}(1-s) - \binom{n-1}{d} \right] (1-s)^{n-d-1}L(s)^{d-1} \,\dint s \\
		&>\left( \frac{L(d/n)}{d/n} \right)^{d-1} \binom{n}{d} \int \limits_0^1 (1-s)^{n-d-1} s^{d-1} \left( (1-s) - \frac{n-d}{n} \right) \dint s \\
		&=\left( \frac{L(d/n)}{d/n} \right)^{d-1} \binom{n}{d} \left( \B(d,n-d+1) - \frac{n-d}{n} \B(d,n-d) \right) \\
		&=0 \mbox{,}
	\end{align*}
	where we used the well-known relation $\B(d,n-d+1)=\frac{n-d}{n} \B(d,n-d)$ for the beta function.

\medskip

As the next case we consider the class $\bB$ of beta-type distribution on the unit ball $\BBd$ with density $f_{\bB,\beta}$ for some $\beta>-1$. In this case the proof follows almost line by line the proof for $\bH$, up to some minor modifications. In particular, \eqref{e3} stays the same except that now $L(s)=f\left(F^{-1}(s)\right)  \sqrt{ 1-(F^{-1}(s))^2 } $, where $F(h)=F_{\bB,\beta}(h)$ and $f(h)=f_{\bB,\beta}(h)$. Therefore, it follows that
	\[ L^{\prime\prime}(s) = -\frac{2}{c_{\bB,1,\b}}  \left( \b + \frac{d}{2} \right) \left( 1-(F^{-1}(s))^2 \right)^{-\b-1-\frac{d}{2}}\mbox{,} \]
where the constant $c_{\bB,1,\b}$ is $c_{\bB,1,\beta}:=\pi^{-1/2}\Gamma(\beta+{3\over 2})\Gamma(\beta+1)^{-1}$. Since $F^{-1}(s) \in (-1,1)$, we obtain $L^{\prime\prime}(s)<0$ and can conclude as in the proof for the class $\bH$ presented above.
	
\medskip

Finally, we consider the case of the uniform distribution on $\SSd$. Here we get by applying the spherical Blaschke-Petkantschin formula from \Cref{thm:spherical-affine-BP} and using the abbreviations $F(h)=F_{\bU}(h)$ and $f(h)=f_{\bU}(h)$,
	\begin{align*}
	&\PP(\conv(X_1,\dots,X_d) \text{ is a facet of }P_n ) \\ 
	&=c \int \limits_{A(d,d-1)} \int \limits_{(H\cap\SSd)^d} \left( (1-F(h))^{n-d} + F(h)^{n-d} \right) \Delta_{d-1}(x_1,\dots,x_d) (1-h^2)^{-\frac{d}{2}} \\ 
	& \qquad\qquad \times \mathcal{H}^{d(d-2)}_{(H\cap\SSd)^d}(\dint(x_1,\ldots,x_d)) \, \mu_{d-1}(\dint H) \\
	&=c \int \limits_{\SSd} \int \limits_0^1  \int \limits_{(H\cap \SSd)^d} \left( (1-F(h))^{n-d} + F(h)^{n-d} \right) \Delta_{d-1}(x_1,\dots,x_d) (1-h^2)^{-\frac{d}{2}}   \\ 
	& \qquad\qquad \times \mathcal{H}^{d(d-2)}_{(H\cap\SSd)^d}(\dint(x_1,\ldots,x_d)) \, \dint h \, \mathcal{H}^{d-1}_{\SSd}(\dint u) \\
	& =c \int \limits_{\SSd} \int \limits_0^1 \left( (1-F(h))^{n-d} + F(h)^{n-d} \right) (1-h^2)^{d\frac{d-2}{2}+\frac{d-1}{2}-\frac{d}{2}} \, \dint h \, \mathcal{H}^{d-1}_{\SSd}(\dint u) \\
	&\qquad\qquad\times \int \limits_{(\SS^{d-2})^d}  \Delta_{d-1}(y_1,\dots,y_d) \, \mathcal{H}^{d(d-2)}_{(\SS^{d-2})^d}(\dint(y_1,\ldots,y_d)) \mbox{,}
	\end{align*}
	where the substitution $x_i=y_i \sqrt{1-h^2}$ with $\mathcal{H}^{d-2}_{H\cap \SSd}(\dint x_i) = (1-h^2)^\frac{d-2}{2} \mathcal{H}^{d-2}_{H\cap\SSd}(\dint y_i)$ was used. In particular, this transforms the integration over $(H\cap\SSd)^d$ into a $d$-fold integral over the unit sphere in $H$, which in turn has been identified with $\SS^{d-2}$ due to rotational invariance. Since the integral over $(\SS^{d-2})^d$ is a known positive constant only depending on $d$ (the precise value can be deduced from \cite[Theorem 8.2.3]{SW}, for example), we get by rotational invariance of the underlying distribution that
	\begin{align*}
	&\PP(\conv(X_1,\dots,X_d) \text{ is a facet of }P_n )\\
	 & =c \int \limits_0^1 \left( (1-F(h))^{n-d} + F(h)^{n-d} \right) (1-h^2)^{d\frac{d-3}{2}+\frac{d-1}{2}} \,\dint h \\
	& = c \int \limits_{-1}^1 (1-F(h))^{n-d} f(h)^d (1-h^2)^{\frac{d-1}{2}} \,\dint h \mbox{.}
	\end{align*}
	As a consequence, also for the uniform distribution on $\SSd$ we arrive at an expression of the form \eqref{e3}, this time with $L(s)=f(F^{-1}(s))\sqrt{1-(F^{-1}(s))^2}$. From this point on, the proof can be completed as in the case of the distribution class $\bH$ or $\bB$. This completes the argument.
\end{proof}

\subsection{Discussion}

Let $p:\RR^d\to[0,\infty)$ denote a probability density. By a careful inspection of the proof of Theorem \ref{thm:monotonicity} we see that the following properties of the density $p$ have been used there. First of all, we used that $p$ is spherically symmetric, that is, $p(x)$ only depends on $x=(x_1,\ldots,x_d)\in\RR^d$ via $\|x\|$.  By abuse of notation, we shall write $p(r):(0,\infty)\to[0,\infty)$ with $r^2=x_1^2+\ldots+x_d^2$ for the radial part of the density $p$. This was essential to apply the Blaschke-Petkantschin formulae, which use the invariant hyperplane measure $\mu_{d-1}$. Moreover, given $H\in A(d,d-1)$ with distance $h$ to the origin, we have used that we can find $\varphi(h),\psi(h)>0$ such that
\begin{equation}\label{eq:Scaling}
p(\sqrt{r^2+h^2}) = \varphi(h)\,p\Big({r\over \psi(h)}\Big)
\end{equation}
for all $r>0$. For example, for the density $p_{\bH,\beta}$, $\beta>d/2$,  the scaling property \eqref{eq:Scaling} is satisfied with $\varphi(h)=(1+h^2)^{-\beta}$ and $\psi(h)=\sqrt{1+h^2}$, see \eqref{eq:Pythagoras}. This scaling property was essential when we separated what happens within $H$ from the contribution that arises from the distance of $H$ to the origin. However, all rotationally symmetric densities with (almost everywhere differentiable) radial part satisfying the scaling property \eqref{eq:Scaling} with an (almost everywhere differentiable) function $\psi$ have been classified by Miles \cite{Miles} (see p.\ 376 there) and Ruben and Miles \cite{RubenMiles}. They precisely correspond to the distributions in the classes $\bG$, $\bH$, $\bB$ as well as to the exceptional distributions in $\bU$, for which Theorem \ref{thm:monotonicity} is formulated. % In other words, these classes comprise precisely the distributions on $\RR^d$ for which the monotonicity question raised in the introduction can be answered positively by means of method we used.

On the other hand, this does not mean that $\bG$, $\bH$, $\bB$ and $\bU$ contain the only rotationally symmetric distributions on $\RR^d$ for which such computations are possible. For example, the density with radial part $p_{\beta,j}(r)=c_{\beta,d,j}\,r^{2j}/(1+r^2)^\beta$, $r>0$, $j\in\{0,1,2,\ldots\}$ and $\beta>j+d/2$, which does not belong to the class $\bH$ whenever $j>0$, satisfies the following generalization of the scaling property \eqref{eq:Scaling}:
\begin{align*}
p_{\beta,j}(\sqrt{r^2+h^2}) &= \sum_{k=0}^j \varphi_{k}(h)p_{\beta,k}\Big({r\over\psi(h)}\Big) %\varphi_1(h)p_{\beta,1}\Big({r\over\psi(h)}\Big)+\varphi_2(h)p_{\beta,0}\Big({r\over\psi(h)}\Big)
\end{align*}
with
\begin{align*}
\varphi_{k}(h)={j\choose k}h^{2(k-j)}(1+h^2)^{-\beta}\quad (k=0,\ldots,j)\qquad\text{and}\qquad\psi(h) = \sqrt{1+h^2}.
%\varphi_1(h)=(1+h^2)^{-(\beta-1)},\qquad\varphi_2(h)=h^2(1+h^2)^{-(\beta-1)}\qquad\text{and}\qquad\psi(h) = \sqrt{1+h^2}.
\end{align*}
One can check that the $1$-dimensional marginal density of $p_{\beta,j}$ equals
$$
f_{\beta,j}(x_1) = \sum_{k=0}^j{j\choose k}{c_{\beta,d,j}\over c_{\beta,d-1,k}}x_1^{2(k-j)}(1+x_1^2)^{k+{d-1\over 2}-\beta},
$$
and that from here on the argument based on the affine Blaschke-Petkantschin formula can be applied term-by-term. Unfortunately, the computations in such and similar situations become quite involved. Moreover, to classify \textit{all} rotationally symmetric densities on $\RR^d$ for which these computations can be performed seems to be out of reach.

One might also ask whether the method based on  Blaschke-Petkantschin formulae yields monotonicity of the mean facet number in such situations where the random points $X_1,\ldots,X_n$ are independent with distributions belonging to one of the classes $\bG$, $\bH$, $\bB$ and $\bU$, but not necessarily the same (a so-called mixed case). That is, some of the $X_i$'s are Gaussian, some distributed according to a distribution in $\bH$ etc.\ (but within each class we choose every time the same scale parameter $\sigma$). Unfortunately, this does not work and, in fact, the method breaks down. The reason is that each distribution class requires its individual substitution, which is adapted to its respective scaling property \eqref{eq:Scaling}. The resulting different rescalings in the hyperplane $H$ distort the relationship between the $(d-1)$-volume in $H$ before and after the transformation, cf.\ \cite{RubenMiles}. %On the other hand the method does give monotonicity of the mean facet number if the random points are all Gaussian, but with possibly different scale parameters (variances). The reason is that in such a case, the substitutions all have the same effect on the $(d-1)$-dimensional volume measure in $H$. However, this is no more true for the other distribution classes $\bH$, $\bB$ and $\bU$, cf.\ \cite{RubenMiles}.

\section{Random convex hulls on a half-sphere}\label{sec:halfsphere}

In this section we consider an application of Theorem \ref{thm:monotonicity} to convex hulls generated by random points on a half-sphere. We fix $d\geq 2$, denote by $\SS^d$ the $d$-dimensional unit sphere in $\RR^{d+1}$ and define the half-sphere
\[
\SS^d_+=\{y=(y_1,\ldots,y_{d+1})\in \SS^d : y_{d+1}>0\}\,.
\]
Furthermore, we let $\bS$ be the class of probability distributions on $\SS_+^d$ that have density
\[
p_{\bS,\alpha}(y)=c_{\bS,\alpha}\,y_{d+1}^{\,\alpha},\qquad y=(y_1,\ldots,y_{d+1})\in\SS_+^d, \quad\alpha>-1,
\]
with respect to the spherical Lebesgue measure on $\SS^d_+$. Here, $c_{\bS,\alpha}>0$ is a suitable normalization constant. In particular, choosing $\alpha=0$ shows that the uniform distribution on $\SS^d_+$ belongs to the class $\bS$.

For fixed $\alpha>-1$ and $n\geq d+1$ we let $X_1,\ldots,X_n$ be independent random points that are distributed on $\SS_+^d$ according to the density $p_{\bS,\alpha}$. By $S_n$ we denote the {\em spherical} convex hull of $X_1,\ldots,X_n$, that is, the smallest spherically convex set in $\SS_+^d$ containing the points $X_1,\ldots,X_n$. For the special choice $\alpha=0$, this model has recently been studied in \cite{BaranyHugReitznerSchneider}. In particular, it is shown in \cite{BaranyHugReitznerSchneider} that for this choice of $\alpha$ the mean number of facets $\EE f_{d-1}(S_n)$ of the spherical random polytope $S_n$ converges to a finite constant only depending on $d$, as $n\to\infty$ (a similar result is in fact valid for all distributions in $\bS$, see \cite{AldousEtAl,Carnal,EddyGale}). As a special case, our next result shows the somewhat surprising fact that this limit is approached in a strictly monotone way.

\begin{thm}
Let $X_1,\ldots,X_n$, $n\geq d+1$, be independent and identically distributed according to a probability measure belonging to the class $\bS$. Then,
\[
\EE f_{d-1}(S_n) > \EE f_{d-1}(S_{n-1}).
\]
\end{thm}
\begin{proof}
Let \(g\colon\RR^d\to \SS^d_+\) be the mapping defined as
	\begin{equation*}
		g(x)=\Biggl(\frac{x_1}{\sqrt{1+\norm{x}^2}},\ldots,\frac{x_d}{\sqrt{1+\norm{x}^2}},\frac{1}{\sqrt{1+\norm{x}^2}}\Biggr),
	\end{equation*}
with inverse given by
	\begin{equation*}
		g^{-1}(y)=\biggl(\frac{y_1}{y_{d+1}},\ldots,\frac{y_d}{y_{d+1}}\biggr)
	\end{equation*}
	(this is known as the gnomonic projection). Let \(\mathrm{D}g\) be the Jacobian matrix of \(g\) and put \(J_g(x)\coloneqq\sqrt{\det{\mathrm{D}g(x)}^{\!\intercal}\mathrm{D}g(x)}\). Then, it holds that
		\begin{equation*}
	 		J_g(x)=(1+\norm{x}^2)^{-\frac{d+1}{2}},
		\end{equation*}
see \cite[Proposition 4.2]{BesauWerner}. Moreover, for a measurable subset $A\subset\RR^d$ and a measurable function \(f\colon A\to\RR \) the area formula \cite[Theorem 3.2.3]{Federer} says that
	\begin{equation*}
	    \int_{A}  f(x)\,\dint x=\int_{g(A)} f\circ g^{-1}(y)(J_g\circ g^{-1}(y))^{-1}\,\cH^d_{\SS^d_+}(\dint y).
	\end{equation*}
Next, we notice that \(1+\norm{g^{-1}(y)}^2=y_{d+1}^{-2}\) and apply the formula with $f(x)=p_{\bH,\beta}(x)$ for some $\beta>d/2$:
		\begin{equation*}
		\int_{A}  c_{\bH,d,\beta}\,(1+\norm{x}^2)^{-\beta}\,\dint x=\int_{g(A)} c_{\bH,d,\beta}\, y_{d+1}^{2\beta-d-1}\,\cH^d_{\SS^d_+}(\dint y),
		\end{equation*}
where $c_{\bH,d,\beta}=\pi^{-d/2}\Gamma(\beta)/\Gamma(\beta-{d\over 2})$ is the normalization constant of the density $p_{\bH,\beta}$ defined in \eqref{eq:DensityH}. As a result, we see that the density \(p_{\bS,2\beta-d-1}\) on $\SS_+^d$ is the push-forward of the density \(p_{\bH,\beta}\) on $\RR^d$ under $g$. Note also that \(2\beta-d-1>-1\) since \(\beta>d/2\) and that the uniform measure on the half-sphere corresponds to the choice \(\beta=(d+1)/2\).

The above discussion shows the following. Let $P_n$ be the random convex hull in $\RR^d$ generated by $n$ independent points with density $p_{\bH,\beta}$. Then, the push-forward of $P_n$ has the same distribution as the spherical random polytope $S_n$ with \(\alpha=2\beta-d-1\). Moreover, the facets of $P_n$ are in one-to-one correspondence with those of $S_n$. As a consequence, the mean facet number of the spherical random polytope $S_n$ is the same as the mean facet number of the random convex hull $P_n$, i.e.,
\[\EE f_{d-1}(S_n)=\EE f_{d-1}(P_n).\]
Thus, the monotonicity follows from Theorem \ref{thm:monotonicity}.
\end{proof}

\subsection*{Acknowledgement}
We would like to thank Zakhar Kabluchko (M\"unster) for a number of useful conversations and for bringing references \cite{EddyGale} and \cite{RubenMiles} to our attention. Thanks go also to Matthias Reitzner (Osnabr\"uck) for his valuable remarks on an earlier version of this text.\\ JG, DT and NT have been supported by the Deutsche Forschungsgemeinschaft (DFG) via RTG 2131 {\em High-dimensional Phenomena in Probability -- Fluctuations and Discontinuity}.

\end{document}